\documentclass[12pt,reqno,oneside]{amsart}
\usepackage{amsmath}
\usepackage{amssymb}
\usepackage{amsthm}
\usepackage{graphics}
\usepackage{eucal}
\usepackage{bm}
\usepackage[colorlinks]{hyperref}
\usepackage[capitalize,nameinlink,noabbrev]{cleveref}
\usepackage{mathrsfs}
\usepackage[normalem]{ulem}
\usepackage{color}
\usepackage[dvipsnames]{xcolor}
\usepackage{mathtools}
\usepackage{geometry}
\newtheorem{theorem}{Theorem}[section]
\newtheorem{definition}{Definition}[section]

\newtheorem{lemma}{Lemma}[section]

\newtheorem{remark}{Remark}[section]

\numberwithin{equation}{section}

\newcommand{\GL}{\operatorname{GL}}

\newcommand{\supp}{\operatorname{supp}}
\newcommand{\cov}{\operatorname{cov}}

\newcommand{\bq}{\mathbf{q}}
\newcommand{\bx}{\mathbf{x}}

\newcommand{\cL}{\mathcal{L}}
\newcommand{\bp}{\mathbf{p}}

\newcommand{\by}{\mathbf{y}}

\newcommand{\Q}{\mathbb{Q}}

\newcommand{\Z}{\mathbb{Z}}

\newcommand{\R}{\mathbb{R}}
\newcommand{\N}{\mathbb{N}}

\newcommand\abs[1]{\left\vert#1\right\vert}
\newcommand\norm[1]{\left\lVert#1\right\rVert}

\usepackage{mathtools}
\usepackage{tikz-cd}
\usepackage{graphicx}

\begin{document}

\title{On absolutely friendly measures on $\Q_S^d$  }

\author{Shreyasi Datta and Justin Liu}
\address{\textbf{Shreyasi Datta}\\
Department of Mathematics, Uppsala University}
\email{shreyasi1992datta@gmail.com\\
shreyasi.datta@math.uu.se}
\address{\textbf{Justin Liu}\\
Department of Mathematics,University of Michigan, Ann Arbor, MI 48109-1043}
\email{justliu@umich.edu }

\date{} 
\maketitle
\begin{abstract}
    In this paper, we extend the work of Pollington and Velani in \cite{PV} to an $S$-arithmetic set-up, where $S$ is a finite set of valuations of $\Q$.  In particular, for an \textit{absolutely friendly} measure supported on a compact set in $\Q_S^d$, we give a summation condition on approximating function $\psi$ such that $\mu$ almost no point in the compact set is $\psi$ approximable. The crucial ingredient is a version of the simplex lemma that we prove dynamically.
\end{abstract} 
\section{Introduction}

This paper aims to study Diophantine approximation in $\Q_S^d=\prod_{\nu\in S}\Q_\nu^d$, where $S$ is a finite set of valuations of $\Q$, with respect to an \textit{absolutely friendly} measure. When $S=\{\infty\}$, $\Q_\infty:=\R$. To start with, let us recall Dirichlet's theorem in $\R^d.$ For any $\bx \in \R^d$, there exist infinitely many $\bp \in \Z^d$ and $q \in \N$ such that $$\left\Vert\bx - \frac{\bp}{q}\right\Vert_\infty < q^{-(d+1)/d}.$$ Here $\Vert \cdot\Vert_\infty$ is the sup norm in $\R^d.$

Let $\psi:\N\to \R^{+}$ be a monotonically decreasing function. Let us consider the following set \begin{equation}
    \mathcal{W}(\psi):=\left\{\bx\in \R^d~|~\exists \text{ infinitely many } (\bp,q) \in \Z^d \times \N, \text{s.t.} \left\Vert\bx - \frac{\bp}{q}\right\Vert_\infty < \psi(q)\right\}.
\end{equation}
Now let $K$ be a compact set that supports a non-atomic finite measure $\mu$ and $\mathcal{W}_K(\psi):=\mathcal{W}(\psi)\cap K.$ 
When $K=[0,1]^d$ and $\mu=\lambda$ is the Lebesgue measure supported on $K$, using the convergence Borel-Cantelli lemma, it is not hard to see that 
$$\lambda(\mathcal{W}(\psi)\cap {[0,1]^d})=0 \text{ if } \sum_{n=1}^\infty 2^n(2^n\psi(2^n))^d<\infty.$$
 In \cite{PV}, Pollington and Velani considered the class of absolutely friendly measures $\mu$, which includes measures supported
on self-similar sets satisfying the open set condition. A weaker notion of absolutely friendly, called \textit{friendly}, first appeared in \cite{KLW}. One of the main results in \cite[Theorem 1]{PV} gives a summation condition on the approximating function $\psi$ such that $\mu(\mathcal{W}_K(\psi))=0.$ For Cantor’s middle-third set, this theorem first appeared in a work of Weiss, \cite[Theorem 0.1]{W2001}. This summation condition is not optimal, and finding the optimal condition is the content of a potential Khintchine-type theorem for this class of measures, which is currently an open problem. Recently, there have been some serious developments of this problem in the works of \cite{KhLu2023} and  \cite{Yu2023}, where the optimal condition was achieved for suitable fractals satisfying open set condition and with high Hausdorff dimension. All of these problems mentioned here make sense in a more general $S$-arithmetic set-up, and very little is explored in this direction. Our goal for this paper is to initiate that study. 

We first state Dirichlet's theorem in $S$-arithmetic set-up; see \cite[\S 11.2]{KT}. Given $\bx \in B(0,1)^d\subset \Q_S^d$, there exist infinitely many $\tilde\bq=(\bq,q_0) \in \Z^{d+1}\setminus \{0\}$ such that $$\Vert q_0\bx + \bq\Vert^l_S \leq \begin{cases}\frac{1}{\norm{\tilde\bq}^{1+\frac{1}{d}}_\infty},~ \infty\notin S\\
~ \\
 \frac{1}{\vert q_0\vert}^{\frac{1}{d}}_\infty,~ \infty\in S
\end{cases},$$ where $\# S=l.$ Here $\Vert\cdot\Vert_S$ is the max norm that we define in \S \ref{def}. In \cite{KT}, Dirichlet's theorem was stated in the dual setting, but a similar proof gives the above statement. Now, we define the following set;
\begin{equation*}
    W^S(\psi):=\left\{\bx\in\Q_S^d~\left|~\left\Vert \bx+\frac{\bq}{q_0}\right\Vert^l_S\leq \begin{cases} \psi(\Vert\tilde\bq\Vert_\infty), \text{ if } \infty\notin S \\ ~~\\
   \psi(\vert q_0\vert_\infty) \text{ if } \infty\in S
    \end{cases}\text{ for i.m. } \tilde{\bq}=(\bq,q_0)\in\Z^{d+1}\right.\right\}.
\end{equation*} Here i.m. is an abbreviation of \textit{infinitely many}. We call points in $W^S(\psi)$ as $\psi$ approximable and refer the readers to \S \ref{def} for the definition and notation.

The goal is to study $$ W^S(\psi)\cap \supp\mu,$$ where $\mu$ is an absolutely friendly measure in $\Q_S^d.$ Our main theorem is as follows.
\begin{theorem}\label{main1}
Suppose $\mu$ is an $\alpha$-absolutely friendly measure which is supported on a compact set in $\Q_S^d$. Then,
    \begin{equation}\mu(W^S(\psi)\cap \supp\mu)=0 \text{ if } \sum_{n \in \N} (2^{n\frac{d+1}{d}}\psi(2^n))^\alpha < \infty. \end{equation}
\end{theorem}
The above theorem is $S$-arithmetic analogue of the main theorem in \cite[Theorem 1]{PV}. The main ingredient of the proof is a version of the simplex lemma in $\Q_S^d$. One obstruction in this set-up is forming a simplex in the space $\Q_S^d$ that we prove. For even $\Q_p^d$, it was unclear to us how to work with higher dimensional triangles (i.e. simplexes) in that space. So, instead, we understand the problem as a fundamental domain of a lattice in a higher dimensional space. We increase both dimension and valuation. For instance, in the case of $\Q_p^d$, we study the fundamental domain of a lattice (of large covolume) in $\Q_p^{d+1}\times \R^{d+1}.$ 

The proof is dynamical in nature, and we treat $\infty\in S$ and $\infty\notin S$ cases separately. We embed the problem in $\Q_{S,\infty}^d$, where $\Q_{S,\infty}$ is defined in \S \ref{def}. We construct a lattice in  $\Q_{S,\infty}^{d+1}$ whose fundamental domain serves as a `simplex' roughly speaking. Then we show that the volume of the fundamental domain exceeds the volume of the respective ball in the simplex lemma. In \cite{KTV2006}, a proof of the simplex lemma in $\Q_p^2$ was given. Our proof of the simplex lemma differs from theirs, and it is achieved more generally for $\Q_S^n.$ 

\subsection{Notation and definitions}\label{def}

Given some finite set of valuations $S$, we define $$\Q_{S, \infty} := \begin{cases}\begin{aligned} &\prod_{\nu\in S}\Q_\nu \text{ if } \infty\in S,\\ 
&\prod_{\nu\in S}\Q_\nu \times \R \text{ if } \infty\notin S
\end{aligned}\end{cases}$$  and $\Z_S = \{x\in \Q~|~\vert x\vert_\sigma=1, \sigma\notin S\setminus \infty\}$. Note that $\Z_S$ is diagonally embedded in $\Q_{S,\infty}$, and it is well known that this is a lattice in $\Q_{S, \infty}$.

Given $\bx = (\bx_\nu)\in \Q_S^d$, where $\bx_\nu=(x_{\nu,i}) \in \Q_\nu^d$ for $\nu\in S$, and $1\leq i\leq d,$ we denote $\norm{\bx}_S := \max_{\nu\in S}\norm{\bx_\nu}_\nu,$ where $\Vert \cdot\Vert_\nu$ is the sup norm on $\Q_\nu^d.$ In $\Q_S$, we will denote the same norm as $\vert \cdot\vert_S.$ For any $m\geq 1$ integer, suppose $\lambda_\nu$ denotes the normalized Haar measure on $\Q_\nu^m$, where $\lambda_\nu(\Z_\nu^m) = 1$, and $\lambda_S=\prod_{\nu\in S}\lambda_\nu$ is a measure on $\Q_S^m.$ When $\nu=\infty,$ $\lambda_\infty=\lambda$ is the Lebesgue measure, and $\Z_\infty=[0,1)$.
Given $\bx \in \Q_S^d$, we also define the content $c(\bx)$ as $c(\bx) = \prod_{\nu\in S}\norm{\bx_\nu}_\nu.$ 

Let us recall the following definition that was coined in \cite[\S 2]{KLW}, and \cite[\S 1]{PV}. A measure $\mu_\infty$ on $\R^d$ is doubling if for all $\bx \in K$, there exist positive constants $D$ and $r_0$ such that for all $0<r < r_0$, $$
\mu_\infty(B(\bx,2r)) \leq D \mu_\infty (B(\bx,r)).
$$ Here and elsewhere by $B(\bx,r)$ we mean a ball of radius $r$ centered at $\bx$ in an ambient space. If $\bx=(\bx_\nu)_{\nu \in S}\in\Q_S^d$,  then $B(\bx,r)=\prod_{\nu\in S} B_\nu(\bx_\nu,r),$ where $B_\nu(\bx_\nu,r):=B(\bx_\nu,r)$ is a ball of radius $r$ in $\Q_\nu^d$.
For any affine subspace $\mathcal{L}_\nu\subset\Q_\nu^d$, and $\varepsilon>0$, we define $\cL:=\prod_{\nu\in S} \cL_\nu,$ and 
$$
\mathcal{L}^{\varepsilon}:=\{\bx\in\Q_S^d~|~\inf_{\by\in\mathcal{L}}\Vert \bx-\by\Vert_S< \varepsilon\}.
$$
\begin{definition}\label{decaying}
 Let $\mu= \prod_{\nu\in S}\mu_\nu$ be a non-atomic Borel finite measure on $\Q_S^d$, which is supported on $K:=\prod_{\nu\in S} K_\nu \subseteq \Q_S^d$, where each $\mu_\nu$ is supported on compact set $K_\nu$ in $\Q_\nu^d$. Let $\alpha>0$. We say $\mu$ is \textit{absolutely $\alpha$-decaying} if there exist strictly positive constants $C, r_0$ such that for any hyperplane $\cL_\nu \subset \Q_\nu^d$ and any $\varepsilon>0$, we have $$\mu\left(B(\bx,r)\cap \cL^{\varepsilon}\right) \leq C\left(\frac{\varepsilon}{r}\right)^{\alpha l }\mu(B(\bx,r)),$$ for all $\bx \in K$ and $r < r_0$, where $\cL:=\prod_{\nu\in S}\cL_\nu.$
 
\end{definition}
Note that the above definition is equivalent to say that each $\mu_\nu$ is absolutely $\alpha$-decaying as in \cite{PV}. 


\begin{definition}\label{absf}
    A measure $\mu=\prod_{\nu\in S}\mu_\nu$ on $\Q_S^d$ is called absolutely $\alpha$-friendly if $\mu$ is absolutely $\alpha$-decaying, and if $\infty \in S$, then $\mu_\infty$ is also doubling.
\end{definition}

\section{One-dimensional simplex lemma}
The simplex lemma for $\R$ concerns the separation of rationals on the real line. More explicitly, it says that any interval of length $1/2^k, k\in\N,$ can contain at most one rational point $q_0/q$, where $1\leq q< 2^k.$ For $\R^n$, a version of simplex lemma was proved in \cite{PV} and \cite[Lemma 4]{KTV2006}. Apart from the problems we deal with in this paper, the simplex lemma has also been a crucial tool in showing the full Hausdorff dimension for weighted badly approximable vectors; see the survey \cite{BRV}. 

In recent years, there have been many setups where the appropriate simplex lemma has been proved. For instance,  a simplex lemma was proved for rational points on a rational quadric hypersurface in \cite{Klein_DeSax2018} and for projective space $\mathbb{P}^n(\R)$ in \cite{HH2017}. 

In this section, we give proof of one-dimensional simplex lemma for $\Q_S$, whose proof is simpler than the higher dimensional case. Also, the technique of the proof will be different in higher dimensions. 

\begin{lemma}\label{dim1_S_without_infinity}
Let $\infty\notin S$ and $k\in\N$. Given two distinct rational numbers $\frac{m}{n}$ and $\frac{m'}{n'}$, if \begin{equation}\label{eq1}
    \begin{aligned}
        2^k \leq \Vert (m,n)\Vert_\infty < 2^{k+1}, ~~2^k \leq  \Vert (m',n')\Vert_\infty < 2^{k+1},
    \end{aligned}\end{equation}
    then $$\abs{\frac{m}{n} - \frac{m'}{n'}}_S^l > \frac{1}{2^{2k+4}}.$$
\end{lemma}

\begin{proof} Note that $\vert mn'-m'n\vert_\infty\leq 2^{2k+3}$ and $\left\vert \frac{m}{n}-\frac{m'}{n'}\right\vert_S^l\geq \prod_{\nu\in S}\vert \frac{m}{n}-\frac{m'}{n'}\vert_\nu.$ Also,
$$
\prod_{\nu\in S}\left\vert \frac{m}{n}-\frac{m'}{n'}\right\vert_\nu \left\vert  \frac{m}{n}-\frac{m'}{n'}\right\vert_\infty \geq \frac{1}{\prod_{\nu\in S}\vert nn'\vert_\nu \vert nn'\vert_\infty},
$$ since $mn'-m'n\in \Z\setminus \{0\}.$
The above implies 
$$
\left \vert \frac{m}{n}-\frac{m'}{n'}\right\vert_S^l\geq \frac{\vert nn'\vert_\infty}{\vert nn'\vert_\infty \vert mn'-m'n\vert_\infty}>\frac{1}{2^{2k+4}},
$$ because $\vert n\vert_\nu\leq 1, \vert n'\vert_\nu\leq 1$.
\end{proof}

\begin{lemma}\label{dim1_S_with_infinity}
    Let $\infty\in S$. Given two distinct rationals $\frac{m}{n}$ and $\frac{m'}{n'}$, if \begin{equation}\label{eq2}
    \begin{aligned}
        & 2^k \leq \vert n\vert_\infty < 2^{k+1}, ~~~ 2^k \leq  \vert n'\vert_\infty < 2^{k+1},
    \end{aligned}\end{equation}
    then $$\abs{\frac{m}{n} - \frac{m'}{n'}}_S^l > \frac{1}{2^{2k+2}}.$$
\end{lemma}

\begin{proof} 
    Note  \begin{align*}
     \abs{\frac{m}{n} - \frac{m'}{n'}}_S^l\geq  \prod_{\nu\in S\setminus \infty}\abs{\frac{m}{n} - \frac{m'}{n'}}_\nu \abs{\frac{m}{n} - \frac{m'}{n'}}_\infty &\geq \frac{1}{\prod_{\nu\in S\setminus\infty}\abs{nn'}_\nu \abs{nn'}_\infty} \geq \frac{1}{\abs{nn'}_\infty} > \frac{1}{2^{2k+2}}.
    \end{align*}
    
\end{proof}

\section{Proof of Theorem \ref{main1}}
\subsubsection*{Notations:} Given $\mathbf{r}_1,\dots,\mathbf{r}_{d+1} \in \Q^d$, we define a $(d+1) \times (d+1)$ matrix as follows: 
\begin{equation}\label{defnA}
    A(\mathbf{r}_1,\dots,\mathbf{r}_{d+1}) := \begin{bmatrix}
    1 & r_1^{(1)} & \dots & r_d^{(1)} \\
    \vdots & \vdots & \ddots & \vdots \\
    1 & r_1^{(d+1)} & \dots & r_d^{(d+1)}  
\end{bmatrix},
\end{equation}
where each $\mathbf{r}_i = (r_1^{(i)}, \dots, r_d^{(i)})$, $1\leq i\leq d+1$. Then, $  A(\mathbf{r}_1,\dots,\mathbf{r}_{d+1})\subseteq \Q_{S,\infty}^{d+1}$.
\begin{lemma}\label{lattice}
Let $A = A(\mathbf{r}_1,\dots,\mathbf{r}_{d+1})$. Then $A\Z_S^{d+1}$ is a lattice in $\Q_{S,\infty}^{d+1}$ if and only if $\mathbf{r}_1,\dots,\mathbf{r}_{d+1}$ are not on any hyperplane in $\Q^d$. This is equivalent to saying  $\mathbf{r}_1,\dots,\mathbf{r}_{d+1}$ are linearly independent over $\Q.$
\end{lemma}
\begin{proof}
    First, suppose that $\mathbf{r}_1,\cdots,\mathbf{r}_{d+1}$ lie on a line in $\Q^d$. Then, there exists $(c_1,\dots,c_d,b) \in \Q^{d+1}\setminus\{0\}$ such that $c_1x_1 + \dots + c_dx_d +b = 0$. Thus $$c_1\begin{bmatrix}
    r_1^{(1)} \\ \vdots \\ r_1^{(d+1)}
    \end{bmatrix} + \dots + c_d \begin{bmatrix}
    r_d^{(1)} \\ \vdots \\ r_d^{(d+1)}
    \end{bmatrix} + b \begin{bmatrix}
        1 \\ \vdots \\ 1
    \end{bmatrix} = 0.$$
    This in turn means that the columns of $A$ are not linearly independent over $\Q$. So they are not linearly independent over $\Q_{S,\infty}$ by \cite[Lemma 7.1]{KT}. Hence, $A \notin \GL(d+1,\Q_{S,\infty}^{d+1})$. Therefore, $A\Z_S^{d+1}$ is not a lattice in $\Q_{S,\infty}^{d+1}$. The reverse implication follows similarly.
\end{proof}

Let $K=\prod_{\nu \in S}K_\nu$ be as in Definition \ref{decaying}. For $n \in \N$ and $\nu\in S$, let $D_{n,\nu}$ denote a generic ball with center in $K_\nu$ and of radius $$r_n = \frac{1}{6}\frac{1}{2^{(d+2)/dl}} 2^{-\frac{(d+1)}{dl}(n+1)}.$$

We define $D_n:=\prod_{\nu \in S} D_{n,\nu}$. Without loss of generality, we assume that $K\subset \prod_{\nu\in S\setminus\infty}\Z_\nu^d \times [0,1)^d,$ and hence \begin{equation}\label{D}D_n\subset \prod_{\nu\in S\setminus\infty} C_\nu\Z_\nu^d \times [-C_\infty, C_\infty ]^d,\end{equation} for some $C_\nu>0, C_\infty>0.$ 
Note that these constants will depend on the definition of $r_n$.

Let
\begin{equation}
    B_n:=\left \{(\bq,q_0)\in \Z^{d+1}~\left|\begin{cases}
        2^n\leq \Vert (\bq,q_0)\Vert_\infty\leq 2^{n+1} \text{ if } \infty\notin S,\\
        2^n\leq \vert q_0\vert_\infty\leq 2^{n+1} \text{ if } \infty\in S
    \end{cases}\right.\right\}.
\end{equation}
Since $\sum_{n \in \N} (2^{n\frac{d+1}{d}}\psi(2^n))^\alpha < \infty,$ in Theorem \ref{main1} we can assume that for any $c > 0$ and some sufficiently large $n$, \begin{equation}\psi(2^n) < c2^{-\frac{n(d+1)}{d}}.\end{equation}\label{psi2n}
 
From the covering lemma \cite[\S 3]{PV}, and the fact that nonarchimedean balls either contain one another or do not intersect, the following lemma follows.

\begin{lemma}\label{covering1}
There exists a disjoint collection $\mathcal{D}_n$ of balls $D_n$ of radius $r_n$, such that $$K\subset\begin{cases}\begin{aligned}&\bigcup_{D_n \in \mathcal{D}_n} \prod_{\nu\in S\setminus\infty} D_{n,\nu}\times 3D_{n,\infty}, \infty\in S\\
&\bigcup_{D_n \in \mathcal{D}_n} D_n, \infty\notin S\end{aligned}.\end{cases}$$    
\end{lemma}

\subsection{Proof of Theorem \ref{main1} when $\infty \notin S$}

\begin{lemma}\label{intersec1}
    Let $\nu\in S$. Let $\mathcal{D}_n$ be the disjoint collection of balls $D_{n,\nu}$ from Lemma \ref{covering1}. Suppose there exists some $\frac{\bq}{q_0} \in \Q^d$ such that, $B_\nu\left(\frac{\bq}{q_0}, \left(\psi(\norm{\tilde{\bq}}_\infty)\right)^{1/l}\right) \cap D_{n,\nu} \neq \varnothing$, where $\tilde{\bq}=(\bq,q_0)\in B_n$. Then, $\frac{\bq}{q_0} \in D_{n,\nu}$. 
\end{lemma}
\begin{proof}
     Since $\psi$ is monotonically decreasing and $\tilde{\bq}\in B_n$, $ \psi(\norm{\tilde{\bq}}_\infty) \leq \psi(2^n) < r_n^l$ for all large $n.$ If two balls intersect in $\Q_\nu$, one must be contained in the other. Hence the conclusion follows.

\end{proof}
\begin{lemma}\label{lem1}
    Let $A= A(\frac{\bp_1}{q_1},\cdots,\frac{\bp_{d+1}}{q_{d+1}})$, with $\frac{\bp_i}{q_i}\in\Q^d, i=1,\cdots,d+1$, is defined as in Equation \eqref{defnA}. Suppose 
    \begin{equation} 
    0<\Vert (\bp_i,q_i)\Vert_\infty < 2^{n+1} \text{ for all } i = 1,\cdots, d+1,\end{equation}
     and $A\Z_S^{d+1}$ is a lattice in $\Q_{S,\infty}^{d+1}$. Then
$$\prod_{\nu\in S}\vert \det A\vert_\nu> \frac{1}{2^{(d+1)(n+1)}}.$$
\end{lemma}

\begin{proof}
 By \cite[Theorem 6.1]{KT} we have that the covolume (with respect to the normalized Haar measure $\lambda_{S\cup \infty}$ on $\Q_{S,\infty}^{d+1}$) of $A\Z_S^{d+1}$ in $\Q_{S,\infty}^{d+1}$ is equal to $$c(\det A)=\prod_{\nu\in S}\vert \det A\vert_\nu\vert \det A\vert_\infty.
$$
   Next note that $\det A = \frac{w}{q_1\dots q_{d+1}}$, for some $w\in \Z\setminus \{0\}$, where $\vert w\vert_\infty \leq 2^{(n+1)(d+1)}.$ 
Hence 
$$
\prod_{\nu\in S} \abs{\frac{w}{q_1\dots q_{d+1}}}_\nu \abs{\frac{w}{q_1\dots q_{d+1}}}_\infty > \frac{1}{\vert q_1 \dots q_{d+1}\vert_\infty},
$$ since $\prod_{\nu\in S} \vert w\vert_\nu \vert w\vert_\infty\in \N\setminus \{0\}.$
This implies that 
$$
\prod_{\nu\in S}\vert \det A\vert_\nu> \frac{1}{\vert w\vert_\infty}> \frac{1}{2^{(d+1)(n+1)}}.
$$

\end{proof}

\begin{lemma}\label{fundomain1}
    The fundamental domain of $A\Z_S^{d+1}$ as a lattice in $\Q_{S,\infty}^{d+1}$ is given by $A(\prod_{\nu\in S}\Z_\nu \times [0,1))^{d+1}$.
\end{lemma}
\begin{proof}
Since the fundamental domain of $\Z_S^{d+1}$ in $\Q_{S,\infty}^{d+1}$ is $\prod_{\nu\in S\setminus\infty}\Z_\nu^{d+1} \times [0,1)^{d+1},$ the conclusion of the lemma is straightforward. 
\end{proof}

\begin{lemma}\label{fundomain2} Let $A= A(\frac{\bp_1}{q_1},\cdots,\frac{\bp_{d+1}}{q_{d+1}})$, with $\frac{\bp_i}{q_i}\in D_n, i=1,\cdots,d+1$. Let $\Delta$ be the fundamental domain of the lattice $A\Z_S^{d+1}$ in $\Q_{S,\infty}^{d+1}$. Suppose $(\bp_i,q_i)\in B_n$ for all $i=1,\cdots,d+1.$ Then
    \begin{equation}
       \Delta \subseteq (\prod_{\nu\in S}\Z_\nu \times D_n) \times {A[0,1]}^{d+1} \subseteq \Q_{S\setminus\infty}^{d+1} \times \R^{d+1},
    \end{equation}
\end{lemma}
\begin{proof}
    Using Lemma \ref{fundomain1}, it is enough to show that $A(\prod_{\nu\in S\setminus\infty}\Z_\nu^d \times [0,1)^d]) \subset (\prod_{\nu\in S\setminus\infty}\Z_\nu \times D_n) \times A[0, 1]^{d+1}$.

    This is equivalent to showing that for all $\nu\in S$, $$A\Z_\nu^{d+1} \subset \Z_\nu \times D_{n,\nu}.$$ 
    The containment follows from the definition of the matrix $A$, and $\frac{\bp_i}{q_i}\in D_n$. 

\end{proof}
\begin{lemma}\label{contradiction}
    The hypotheses of Lemma \ref{fundomain2} are never possible. 
\end{lemma}
\begin{proof}
   By Lemma \ref{lem1}, \begin{equation}\label{cov>}
   \prod_{\nu\in S\setminus\infty}\vert \det (A)\vert_\nu\geq \frac{1}{2^{(d+1)(n+1)}}\implies \cov(A\Z_S^{d+1})\geq \frac{1}{2^{(d+1)(n+1)}} \vert \det (A)\vert_\infty. \end{equation}

Suppose $\Delta$ is the fundamental domain of  $A\Z_S^{d+1}$.
   By Lemma \ref{fundomain2},  \begin{equation}
       \Delta \subseteq \left(\prod_{\nu\in S}\Z_\nu \times D_n \right) \times A[0,1]^{d+1}.
    \end{equation}
    But by the definition of $r_n,$ the radius of $D_{n,\nu}$,  \begin{equation}
       \begin{aligned} \cov(A\Z_S^{d+1}) \leq \lambda_{S\cup\{\infty\}}\left(\prod_{\nu\in S}\Z_\nu \times D_n \times A[0,1]^{d+1}\right) \leq &\left(\frac{1}{2^{(d+2)} }2^{-(d+1)(n+1)}\right) \vert \det(A)\vert_\infty \\
        &< \frac{1}{2^{(d+1)(n+1)}}\vert \det(A)\vert_\infty.
        \end{aligned}
    \end{equation}
This contradicts Equation \eqref{cov>}.
\end{proof}

\subsection{Proof of Theorem \ref{main1} when $\infty \in S$}
\begin{lemma}\label{intersec2}
   Let $\mathcal{D}_n$ be the disjoint collection of balls $D_{n,\nu}$ as in Lemma \ref{covering1}. Suppose there exists some $\frac{\bq}{q_0} \in \Q^d$ such that for each $\nu \in S\setminus \infty$, $\underbrace{B_\nu\left(\frac{\bq}{q_0}, (\psi(\abs{q_0}_\infty))^{1/l}\right)}_{\subset \Q_\nu^d} \cap D_{n,\nu}  \neq \varnothing$, and $\underbrace{B_\infty\left(\frac{\bq}{q_0}, (\psi(\abs{q_0}_\infty))^{1/l}\right)}_{\subset\R^d} \cap 3D_{n,\infty}  \neq \varnothing$, where $\tilde{\bq} = (\bq,q_0) \in B_n$. Then, $\frac{\bq}{q_0} \in \left(\prod_{\nu \in S \setminus \infty} D_{n,\nu} \times 6D_{n,\infty}\right)$. 
\end{lemma}
\begin{proof}
    For $\nu = \infty$, fix an arbitrary point $\bx \in B_\infty\left(\frac{\bq}{q_0}, (\psi(\abs{q_0}_\infty))^{1/l}\right) \cap 3D_{n,\infty}$. Furthermore, we denote the center of $3D_{n,\infty}$ as $k_{n,\infty}$. 
    
    Without loss of generality, we can similarly assume (by picking some sufficiently large $n$) that $(\psi(\abs{q_0}_\infty))^{1/l} <  r_n$. So, we therefore have that $\Vert\bx - \frac{\bq}{q_0}\Vert_\infty < r_n$. Additionally, by the definition of $D_n$, we have $\norm{k_{n,\infty}-\bx}_\infty < 3r_n$. So, by the triangle inequality, \begin{equation}
        \norm{\frac{\bq}{q_0} - k_{n,\infty}}_\infty < 4r_n<6r_n,
    \end{equation}
    and so, $\frac{\bq}{q_0} \in 6D_{n,\infty}$ as desired.
    
    Similarly, by stronger triangle inequality  for $\nu\in S\setminus\infty,$        $B_\nu\left(\frac{\bq}{q_0}, \left(\psi(\abs{q_0}_\infty)\right)^{1/l}\right) \subseteq D_{n,\nu}$ and hence $\frac{\bq}{q_0} \in D_{n,\nu}$. 
\end{proof}

\begin{lemma}\label{lem2}
    Let $A= A(\frac{\bp_1}{q_1},\cdots,\frac{\bp_{d+1}}{q_{d+1}})$, with $\frac{\bp_i}{q_i}\in\Q^d, i=1,\cdots,d+1$, is defined as in Equation \eqref{defnA}. Suppose     
   \begin{equation}\label{height_condition} 
    2^n <\abs{q_i}_\infty < 2^{n+1} \text{ for all } i = 1,\cdots,d+1,\end{equation}
    
    and $A\Z_S^{d+1}$ is a lattice in $\Q_{S,\infty}^{d+1} = \Q_S^{d+1}$. Then,
    $$\cov(A\Z_S^{d+1}) > \frac{1}{2^{(d+1)(n+1)}}.$$
\end{lemma}
\begin{proof}
   Firstly, we know that $$\cov(A\Z_S^{d+1}) = \prod_{\nu \in S \setminus \infty} \abs{\det A}_\nu \abs{\det A}_\infty,$$  with respect to the normalized Haar measure $\lambda_S$ on $\Q_{S}^{d+1}$.
     We can write $\det A = \frac{w}{q_1\dots q_{d+1}}$ for some $w \in \Z\setminus \{0\}$. Then, \begin{align*}
        \prod_{\nu \in S \setminus \infty} \abs{\det A}_\nu \abs{\det A}_\infty &= \prod_{\nu \in S \setminus \infty} \abs{\frac{w}{q_1\dots q_{d+1}}}_\nu \abs{\frac{w}{q_1\dots q_{d+1}}}_\infty.
\end{align*}
Since $w\in \Z\setminus \{0\}, $ we know $\prod_{\nu\in S} \vert w\vert_\nu \vert w\vert_\infty\geq 1.$ Hence the term above is 
        $$\geq \frac{1}{\prod_{\nu \in S \setminus \infty} \abs{q_1 \dots q_{d+1}}_\nu \abs{q_1 \dots q_{d+1}}_\infty}. $$
        Since $\vert q_i\vert_\nu\leq 1$ for $\nu\in S\setminus\infty$, $i=1,\cdots,d+1.$
    Therefore, by \eqref{height_condition} the term above is 
        $$\geq \frac{1}{\abs{q_1\dots q_{d+1}}_\infty} 
        > \frac{1}{2^{(d+1)(n+1)}}.$$
   
\end{proof}

\begin{lemma}\label{fundomain_with_infinity} Let $A= A(\frac{\bp_1}{q_1},\cdots,\frac{\bp_{d+1}}{q_{d+1}})$, with $\frac{\bp_i}{q_i}\in \prod_{\nu \in S \setminus \infty} D_{n,\nu} \times 6D_{n,\infty}, i=1,\cdots,d+1$. Let $\Delta$ be the fundamental domain of the lattice $A\Z_S^{d+1}$ in $\Q_{S,\infty}^{d+1}$. Suppose $(\bp_i,q_i)\in B_n$ for all $i=1,\cdots,d+1.$ Then,\begin{equation}
        \Delta \subseteq (\prod_{\nu \in S \setminus \infty} \Z_\nu \times \prod_{\nu \in S \setminus \infty} D_{n,\nu}) \times \left([0,1] \times 6D_{n,\infty}\right) \subseteq \Q_{S\setminus\infty}^{d+1} \times \R^{d+1}.
    \end{equation}
\end{lemma}
\begin{proof}
    Firstly, using Lemma \ref{fundomain1}, it is enough to show $\Delta=A(\prod_{\nu\in S\setminus\infty}\Z_\nu \times [0,1))^{d+1} \subset \prod_{\nu\in S\setminus\infty}(\Z_\nu \times D_{n,\nu}) \times [0,1]\times 6D_{n,\infty}. $
   
    For $\nu \in S \setminus \infty$, since $\frac{\bp_i}{q_i} \in D_{n,\nu}$ for $i = 1,\dots,d+1$, $$A\Z_\nu^{d+1} \subset \Z_\nu \times D_{n,\nu}.$$
    For the $\infty \in S$ case, similarly, since $\frac{\bp_i}{q_i} \in 6D_{n,\infty}$ for $i=1,\dots,d+1$,
    
    $$A[0,1]^{d+1} \subset [0,1] \times 
    6D_{n,\infty}.$$
\end{proof}

\begin{lemma}\label{contradiction2}
    The hypotheses of Lemma \ref{fundomain_with_infinity} are never possible. 
\end{lemma}
\begin{proof}
    By Lemma \ref{lem2}, \begin{equation}\label{cov>>}
        \cov(A\Z_S^{d+1}) = \prod_{\nu \in S \setminus \infty} \abs{\det A}_\nu \abs{\det A}_\infty > \frac{1}{2^{(d+1)(n+1)}}.
    \end{equation}

Suppose $\Delta$ is the fundamental domain of  $A\Z_S^{d+1}$. By Lemma \ref{fundomain_with_infinity}, \begin{equation}
    \Delta \subseteq (\prod_{\nu \in S \setminus \infty} \Z_\nu \times \prod_{\nu \in S \setminus \infty} D_{n,\nu}) \times [0,1] \times 6D_{n,\infty}.
\end{equation}
Also note, by the definition of radius of $D_n$,
\begin{equation}
    \begin{aligned}
        \lambda_S \left(\prod_{\nu \in S \setminus \infty} \Z_\nu \times \prod_{\nu \in S \setminus \infty} D_{n,\nu} \times [0,1] \times 6D_{n,\infty}\right) &\leq \frac{1}{2^{d+2}} 2^{-(d+1)(n+1)} \\
        &< \frac{1}{2^{(d+1)(n+1)}}.
    \end{aligned}
\end{equation}
This contradicts Equation \eqref{cov>>}.

\end{proof}

\subsection{Finishing the proof}
From Lemma \ref{contradiction} and Lemma \ref{contradiction2} the following lemma follows, since we cannot have $d+1$ linearly independent elements of $D_n$ if $\infty \notin S$ (or $\prod_{\nu \in S \setminus \infty} D_{n,\nu} \times 6D_{n,\infty}$ if $\infty \in S$), with height restrictions as in $B_n$.


\begin{lemma}\label{simplex}
Every rational $\frac{\bq}{q_0}\in \begin{cases}\begin{aligned} &D_n \text{ if } \infty\notin S,\\
 &\prod_{\nu\in S\setminus \infty} D_{n,\nu}\times 6D_{n,\infty} \text{ if } \infty\in S\end{aligned}\end{cases}$, with $(\bq,q_0)\in B_n$ must lie in 
some $\cL(D_n)_{\Q}:=\{\bx\in \Q^{d}~|~ c_1x_1+\cdots+c_dx_d=b\},$ for some $(c_1,\cdots,c_d,b)\in \Z^{d+1}\setminus \{0\}.$
\end{lemma}

Let us define $\mathcal{L}(D_n)_\nu:=\{\bx=(x_{\nu,i})\in \Q_\nu^d~|~c_1x_{1,\nu}+\cdots+c_d x_{\nu,d}=b\}$,
for every $\nu\in S\cup\{\infty\}.$ Also, define $\mathcal{L}(D_n):=\prod_{\nu\in S}\mathcal{L}(D_n)_\nu$.\\
\subsubsection{Proof for $\infty\notin S$} Let us consider $\infty\notin S.$ For $n \in \N$, we define \begin{equation}
        A_n := \bigcup_{\tilde\bq\in B_n} \prod_{\nu\in S}B_\nu\left(\frac{\bq}{q_0}, \left(\psi(\norm{\tilde{\bq}}_\infty)\right)^{1/l}\right),
    \end{equation} where $\tilde\bq=(\bq, q_0).$
    By definition $W^S(\psi)\cap K = \limsup_{n \to \infty} A_n \cap K$. If $A_n\cap D_n\neq\emptyset,$ then there exists some $\tilde{\bq}\in B_n$ such that for each $\nu\in S,$
    $$
    B_\nu\left(\frac{\bq}{q_0}, \left(\psi(\norm{\tilde{\bq}}_\infty)\right)^{1/l}\right) \cap D_{n,\nu}\neq\emptyset.
    $$
    Hence by Lemma \ref{intersec1}, for some sufficiently large $n$ we have $\frac{\bq}{q_0}\in D_{n,\nu}$ for all $\nu \in S$, which implies $\frac{\bq}{q_0}\in D_n$. Since $\tilde\bq\in B_n,$ and  $\frac{\bq}{q_0}\in D_n$ Lemma \ref{simplex} implies $\frac{\bq}{q_0}\in \cL(D_n)$. Hence for every $\nu\in S$, $B_\nu\left(\frac{\bq}{q_0}, \left(\psi(\norm{\tilde{\bq}}_\infty)\right)^{1/l}\right)\subset B_\nu\left(\frac{\bq}{q_0}, \left(\psi(2^n)\right)^{1/l}\right)\subset \cL(D_n)_\nu^\varepsilon,$ where $\varepsilon= \left(\psi(2^n)\right)^{1/l}.$ Thus $$
    \prod_{\nu\in S} B_\nu\left(\frac{\bq}{q_0}, \left(\psi(\norm{\tilde{\bq}}_\infty)\right)^{1/l}\right)\subset \prod_{\nu\in S}\cL(D_n)_\nu^\varepsilon.$$ Note $\cL(D_n)^\varepsilon=\prod_{\nu\in S} \cL(D_n)_\nu^\varepsilon.$
Therefore, we have 
\begin{equation}\label{Ainsideaffine}
    A_n\cap D_n\subset \cL(D_n)^\varepsilon \cap D_n.
\end{equation}

    Thus, \begin{align*}
        \mu(A_n \cap K) \stackrel{Lemma ~\ref{covering1}}{\leq}  \mu\left(A_n \cap \bigcup_{D_n \in \mathcal{D}_n}D_n\right) 
        &\leq \sum_{D_n \in \mathcal{D}_n} \mu(A_n \cap D_n) \\
        &\stackrel{\eqref{Ainsideaffine}}{\leq}  \sum_{D_n \in \mathcal{D}_n} \mu(\cL(D_n)^{(\varepsilon)} \cap D_n) \\
        &\stackrel{Definition ~\ref{decaying}}{\ll}  \sum_{D_n \in \mathcal{D}_n} (2^{n\frac{d+1}{dl}}\psi(2^n)^{1/l})^{\alpha l}\mu(D_{n}).
        \end{align*}
Since the collection $\mathcal{D}_n$ is disjoint, we can bound the above term  as follows,
        \begin{align*}
        &\leq (2^{n\frac{d+1}{d}}\psi(2^n))^\alpha  \mu(K) \ll (2^{n\frac{d+1}{d}}\psi(2^n))^\alpha.
    \end{align*}
  
    Therefore, we have that \begin{align*}
        \sum_{n \in \N} \mu(A_n \cap K) < \infty \iff \sum_{n \in \N} (2^{n\frac{d+1}{d}}\psi(2^n))^\alpha < \infty.
    \end{align*}
    The conclusion now follows from the convergence Borel-Cantelli lemma.

    \subsubsection{Proof for $\infty\in S$} The final part of the proof in the case $\infty\in S$, will be very similar to the previous case. For $n \in \N$, we define \begin{equation}
        A_n := \bigcup_{\tilde\bq\in B_n} \prod_{\nu\in S}B_\nu\left(\frac{\bq}{q_0}, \left(\psi(\vert q_0\vert_\infty)\right)^{1/l}\right).
    \end{equation} If $A_n\cap \bigcup_{D_n \in \mathcal{D}_n} \prod_{\nu \in S \setminus \infty} D_{n,\nu} \times 3D_{n,\infty} \neq\emptyset,$ then there exists some $\tilde{\bq}=(\bq,q_0)\in B_n$ such that for each $\nu\in S\setminus \infty,$
    $$
    B_\nu\left(\frac{\bq}{q_0}, \left(\psi(\vert q_0\vert_\infty)\right)^{1/l}\right) \cap D_{n,\nu}\neq\emptyset, B_\infty\left(\frac{\bq}{q_0}, \left(\psi(\vert q_0\vert_\infty)\right)^{1/l}\right) \cap 3D_{n,\infty}\neq\emptyset.
    $$ Then, by Lemma \ref{intersec2}, $\frac{\bq}{q_0} \in \left(\prod_{\nu \in S \setminus \infty} D_{n,\nu} \times 6D_{n,\infty}\right)$. Hence by Lemma \ref{simplex} and similar argument as before $$
    \prod_{\nu\in S} B_\nu\left(\frac{\bq}{q_0}, \left(\psi(\vert q_0\vert_\infty)\right)^{1/l}\right)\subset \prod_{\nu\in S}\cL(D_n)_\nu^\varepsilon,$$ where $\varepsilon= \left(\psi(2^n)\right)^{1/l}.$ 
    Thus, \begin{align*}
        \mu(A_n \cap K) &\stackrel{Lemma ~\ref{covering1}}{\leq}  \mu\left(A_n \cap \bigcup_{D_n \in \mathcal{D}_n} \prod_{\nu \in S \setminus \infty} D_{n,\nu} \times 3D_{n,\infty} \right) 
        \\ &\leq \sum_{D_n \in \mathcal{D}_n} \mu(A_n \cap \prod_{\nu \in S \setminus \infty} D_{n,\nu} \times 3D_{n,\infty}) \\
        &{\leq}  \sum_{D_n \in \mathcal{D}_n} \mu(\cL(D_n)^{(\varepsilon)} \cap \prod_{\nu \in S \setminus \infty} D_{n,\nu} \times 3D_{n,\infty}). 
        \end{align*}
Since $\mu$ is a product measure the above is,
        \begin{align*}
         &\leq \sum_{D_n \in \mathcal{D}_n} \prod_{\nu\in S\setminus \infty}\mu_\nu(\cL(D_n)_\nu^{(\varepsilon)} \cap D_{n,\nu}) \mu_\infty(\cL(D_n)_\infty^{(\varepsilon)} \cap 4D_{n,\infty}). 
         \end{align*}
        By decaying property in Definition \ref{decaying}, the above is 
        
         \begin{align*}
        &\stackrel{Definition ~\ref{decaying}}{\ll}  \sum_{D_n \in \mathcal{D}_n} (2^{n\frac{d+1}{dl}}\psi(2^n)^{1/l})^{\alpha l}\prod_{\nu\in S\setminus\infty}\mu_\nu(D_{n,\nu})\mu_\infty(4D_{n,\infty}). 
        \end{align*}
        Since $\mu_\infty$ is doubling as in Definition \ref{absf}, and by disjointness of $\mathcal{D}_n$, the above is, 
        \begin{align*}
        &\stackrel{Definition ~\ref{absf}}{\ll} (2^{n\frac{d+1}{d}}\psi(2^n))^\alpha  \mu(K)\\
        & \ll (2^{n\frac{d+1}{d}}\psi(2^n))^\alpha.
    \end{align*}
    The proof is now complete by the convergence Borel-Cantelli lemma.
   \begin{remark}
     Note that in the $\infty\in S$ case, we only need $\mu_\infty$ to be doubling, whereas we don't need doubling property for $\mu_\nu$ for all $\nu\in S\setminus \infty.$
    \end{remark}
\subsection{Closing remark}
The simplex lemma we proved in this paper can also be used to prove an analogue of \cite[Theorem 2]{PV} in $S$-arithmetic set-up, involving the $s$-dimensional Hausdorff measure. For more details, readers can refer to \cite[Theorem 2]{PV}. 
\subsection{Acknowledgement}
The authors thank support from the University of Michigan where the work was started and Uppsala University where the work was finished. SD is supported by the AMS Simons travel grant, and also by the Knut and Alice Wallenberg Foundation. SD likes to thank Ralf Spatzier for many exciting discussions and support. 

\bibliographystyle{plain}
\bibliography{simplexbib}

\newcommand{\noop}[1]{}
\begin{thebibliography}{10}

\bibitem{BRV}
Victor Beresnevich, Felipe Ramírez, and Sanju Velani.
\newblock {\em Metric Diophantine Approximation: Aspects of Recent Work}, page
  1–95.
\newblock London Mathematical Society Lecture Note Series. Cambridge University
  Press, 2016.

\bibitem{HH2017}
Stephen Harrap and Mumtaz Hussain.
\newblock A note on badly approximabe sets in projective space.
\newblock {\em Math. Z.}, 285(1-2):239--250, 2017.

\bibitem{KhLu2023}
Osama Khalil and Manuel Luethi.
\newblock Random walks, spectral gaps, and {K}hintchine's theorem on fractals.
\newblock {\em Invent. Math.}, 232(2):713--831, 2023.

\bibitem{KT}
D.~Kleinbock and G.~Tomanov.
\newblock Flows on {S}-arithmetic homogeneous spaces and applications to metric
  {D}iophantine approximation.
\newblock {\em Comm. Math. Helv.}, 82:519--581, 2007.

\bibitem{Klein_DeSax2018}
Dmitry Kleinbock and Nicolas de~Saxc\'{e}.
\newblock Rational approximation on quadrics: a simplex lemma and its
  consequences.
\newblock {\em Enseign. Math.}, 64(3-4):459--476, 2018.

\bibitem{KLW}
Dmitry Kleinbock, Elen Lindenstrauss, and Barak Weiss.
\newblock On fractal measures and {D}iophantine approximation.
\newblock {\em Sel. math., New ser.}, 10(4):479 -- 523, 2005.

\bibitem{KTV2006}
Simon Kristensen, Rebecca Thorn, and Sanju Velani.
\newblock Diophantine approximation and badly approximable sets.
\newblock {\em Adv. Math.}, 203(1):132--169, 2006.

\bibitem{PV}
A.~Pollington and S.~Velani.
\newblock Metric diophantine approximation and ‘absolutely friendly’
  measures.
\newblock {\em Selecta Mathematica}, 11, 2005.

\bibitem{W2001}
Barak Weiss.
\newblock Almost no points on a {C}antor set are very well approximable.
\newblock {\em R. Soc. Lond. Proc. Ser. A Math. Phys. Eng. Sci.},
  457(2008):949--952, 2001.

\bibitem{Yu2023}
Han Yu.
\newblock Rational points near self-similar sets.
\newblock \noop{2021}, https://arxiv.org/abs/2101.05910.

\end{thebibliography}
\end{document}